\documentclass{amsart}
\usepackage{amsmath, amsfonts, amsthm,amssymb,hyperref,graphicx,ulem,mathrsfs, ifpdf, verbatim }

\usepackage{fancyhdr} 
\input xy
\xyoption{all}

\newtheorem{theorem}{Theorem}[section]

\theoremstyle{remark}

\numberwithin{equation}{section}

\begin{document}
\title{An Update on the Classification of Rank 2 Weak Fano Threefolds}

\author{Joseph W. Cutrone}
\address{Department of Mathematics, Johns Hopkins University, 3400 N. Charles St, Baltimore, MD 21218}
\email{jcutron2@jhu.edu}

\author{Nicholas A. Marshburn}
\address{Department of Mathematics, Johns Hopkins University, 3400 N. Charles St., Baltimore, MD 21218}
\email{nmarshb1@jhu.edu}

\subjclass[2024]{}

\begin{abstract}
In this paper, an update on the classification of smooth weak Fano threefolds with Picard number two and small anti-canonical maps is given. Geometric constructions are provided for previously open numerical cases by blowing up certain curves on smooth Fano threefolds of Picard number one. This paper provides updated tables in the Appendix and reduces the 14 remaining E1-E* open cases to four. 

\end{abstract}

\maketitle

\section{Introduction}
A smooth Fano variety is a smooth projective variety whose anticanonical class is ample. A smooth weak Fano variety is a smooth projective variety whose anticanonical class is both nef and big. This relaxation of the ampleness condition gives rise to a wider and more complex class of threefolds. Weak Fano threefolds are of interest because they appear in the study of the birational geometry of threefolds, particularly in connection with the Minimal Model Program (MMP). They often serve as building blocks or intermediate steps in understanding the structure of more complicated three-dimensional varieties. Besides the contribution to the classification of Sarkisov links, these weak Fano threefolds have applications to the construction of $G_2$-manifolds and Calabi-Yau threefolds (\cite{CHNP12},\cite{CHNP13}).

In this paper, we provide an update to the numerical classification of smooth weak Fano threefolds $X$ with Picard number two that started with \cite{Tak89}, \cite{JPR05}, \cite{JPR07}, \cite{Kal09} and \cite{CM13}. In these papers, weak Fano threefolds were obtained by blowing up a curve or a point on a smooth Fano threefold of Picard rank 1 under different assumptions. The outcome of these papers was the complete numerical classification of Sarkisov links of special type between smooth Fano threefolds of Picard number one and in turn, the classification of the central models of class rank 2. There are over 200 distinct families of these smooth weak Fano threefolds with Picard number two. The geometric construction of many cases was left open. 

In a series of papers that followed, namely \cite{BL12}, \cite{BL15}, \cite{ACM17}, significant progress was made to complete the geometric classification of these links, which meant to either explicitly construct the Sarkisov link with the known numerical invariants or to show no such link could occur. The geometric classification is mostly done, but not yet complete. For a more in depth summary of the classification scheme for the rank two weak Fano threefold classification problem, see the corresponding sections of \cite{CHNP13}. 

\section{Remarks on the Remaining Open Cases}
The remaining open cases on the E1-E1 table are numbers 28,59,61,80, which are the blow-ups of index one Fano threefolds. 

Cases 59,61,80 also appear on the tables of \cite{JPR07} and the methods to show the link must be of type E1-E1 used in Theorem 3.1 do not produce the required contradiction. It is possible that both links exist with the numerical invariants given: one of type E1-E1 and one of type E1-dP (for cases 59,80) and E1-CB (for case 61). 

Case 28 on the E1-E1 table is the most challenging as the numerical link starts with a del Pezzo threefold $V_2$ of index 2. As $V_2$ does not admit a projective model as a complete intersection in a homogeneous space, the methods of this paper and those of \cite{ACM17} do not apply. 

We leave these cases for future work. 

\subsection{Results}
The main result of this paper is to settle the geometric existence of all but four cases in the original tables of \cite{CM13}. The most updated tables for Sarkisov links of types E*-E* are located in the Appendix. 

\section{Background and Notation}
Let $X$ be a smooth weak Fano threefold of Picard rank two, and $\psi: X\to X'$ its anticanonical morphism. Then either $\psi: X\to X'$ is divisorial, in which case $X'$ is a rank one Gorenstein Fano threefold with canonical non-terminal singularities, or $\psi: X\to X'$ is small, and $X'$ is a rank one non-$\mathbb{Q}$-factorial Gorenstein Fano 3-fold with terminal singularities. These singularities turn out to be isolated cDV singularities and in many cases, only ordinary double points. 

 In addition to the $K_X$-trivial contraction $\psi$, the rank two weak Fano threefold $X$ admits a unique extremal contraction $\phi: X \to Y$, where $Y$ is a (possibly singular) Fano threefold of Picard number 1. By \cite{Ko89}, $\psi$ induces a flop $\chi$ to a smooth Picard rank two threefold $X^+$. $X^+$ is weak Fano with the same anticanonical degree as $X$. The small anticanonical morphism associated to $X^+$ is denoted by $\psi^+$, and the unique extremal contraction is denoted by $\phi^+: X^+ \to Y^+$. This is summarized in the following commutative diagram:
\begin{equation}
\xymatrix{X \ar@{-->}[rr]^{\chi} \ar[d]^{\phi} \ar[dr]^{\psi}& & X^{+} \ar[d]^{\phi^{+}} \ar[dl]_{\psi^{+}} \\
          Y & X' & Y^{+}}
\end{equation}
where $\chi$ is an isomorphism outside of the exceptional locus of $\psi$. 

Mori and Mukai's classification of non-singular threefold extremal rays (\cite{MM84}) along with $\rho=2$ restrict $X$ sufficiently to allow for a complete numerical classification of rank two weak Fano threefolds. In \cite{JPR07}, the authors gave a complete numerical classification in the case that at most one of the Mori contractions $\phi$ or $\phi^+$ contracts a divisor (i.e., is of type E). The authors of this paper, in \cite{CM13}, completed the numerical classification when both Mori contractions are of type E. 

\section{New Constructions}
All varieties are defined over $\mathbb{C}$ and all varieties are assumed to be smooth unless indicated otherwise. 

The numerical classification of Sarkisov links of type E1-E* with small anticanonical morphism can be found in \cite{CM13}. In that original table of 111 numerical links, 11 were proven to be geometrically realisable, 13 not to be geometrically realisable, and 87 numerical links were left open. Later papers (e.g., \cite{ACM17}, \cite{BL12}) tried to complete the geometric classification of these cases and after \cite{ACM17} specifically, there were only twelve open cases on the E1-E1 tables. The remaining open cases on \cite{ACM17} are numbers:  2,10,28,39,47,59,61,67,72,78,80,102. In the following two theorems, we construct the geometric existence for eight of these cases.   

\begin{theorem}
The numerical invariants listed in the E1-E1 table for cases 2,10,39,47,67,72,78, and 102, are all geometrically realizable as a Sarkisov link of Fano threefolds. 
\end{theorem}

\begin{proof}
These cases were originally left open by \cite{ACM17} (Remarks 3.3, 4.3, 6.3) because the authors were not able to show the anticanonical morphism were small. Note that \cite{ACM17} Propositions 3.1, 4.1, and 6.1 settled both the existence of the smooth curve $C$ with given invariants existing on a Fano threefold $Y$ and that the threefold $X$ arising as the blow up of $C$ is weak Fano. The proof below uses the ideas in \cite{CLM19} to show the missing piece: that the anticanonical morphism $\psi: X \to X'$ is small. Once this is proven, the following argument then is applied to show the construction of the Sarkisov link must be that of the case on the E1-E1 table: If the morphism $\phi^+$ were \textit{not} of type E1, then the listed numerical possibilities would have appeared either in [\cite{JPR07}, 7.4, 7.7, p.486] or in the non-E1-E1 tables in \cite{CM13}. Since this is not the case for any of the numerical links listed, $\phi^+$ is of type E1.

In all cases that follow, let $C$ be smooth curve of degree d and genus g contained in a K3 surface $S$ in $Y$. On $X$, let $E$ be the exceptional divisor, $H$ the strict transform of the hyperplane section on $Y$, and $\widetilde S$ the strict transform of $S$. Since $\phi$ restricted to $\widetilde S$ is an isomorphism with $S$, on $\widetilde S$ we denote by $C$ the strict transform of $C$ on $S$. By \cite{Knu02}, the Picard group of $\widetilde S$ is generated by $H_{\widetilde S}$ and $C$.

To show the anticanonical divisor $\psi$ is small, we assume instead that $\psi$ contracts a divisor $D$. Since $H$ and $E$ generate Pic $X$, write $D \sim aH + bE$ for some integers $a$ and $b$ and since $D$ is contracted, $K_X^2 D = 0$.\\

\noindent We now proceed case by case using the notation above:

\begin{enumerate} 
\item[Case 10:] Let $C$ be a smooth curve of degree 4 and genus 1 contained in a K3 surface $S$ in $Y$. Since $D$ is contracted, $K_X^2D = 0$. Using $H^3 = 10, H^2E = 0, HE^2 = -d = -4$, and $E^3 = -rd+2-2g = -4$, we find $0=K_X^2D = (H-E)^2(aH+bE) = 6a+4b$. Solving for integers $a$ and $b$ then gives that $D \sim 2H_E$. Restricting to $\widetilde S$ and using $(H|_{\widetilde S})^2 = 10$, $H_{\widetilde S}C = 4$, and $C^2 = 2g - 2 = 0$, we find $(D|_{\widetilde S})^2 = (2H_{\widetilde S} - 3C)^2 = -8$. If $D$ is contracted to a curve, then $D$ is a conic bundle over a smooth curve $B$ by Proposition 1.8 \cite{JPR05} (whose result is attributed to Wilson \cite{W92}, \cite{W97}, Paoletti \cite{Pa98}, and Minagawa \cite{Mi03}). Since $\widetilde S \in |-K_X|$, we have $D|_{\widetilde{S}} =  a_1l_1 + \ldots + a_nl_n$, where each $l_i$ is a fiber of the conic bundle and the $a_i$ are integers. Since $\rho(\widetilde{S}) = 2$, $\widetilde{S}$  cannot contain two disjoint rational curves, because rational curves on a K3 surface are contractable -2 curves. Hence, $n=1$ and $D|_{\widetilde{S}} = a_1l_1.$  Then $(D|_{\widetilde{S}})^2 = -2a_1^2 = -8$ implies $a_1 = 2$. Then $2l_1 \sim 2H_{\widetilde S}-3C$ implies $l\sim\frac{1}{2}(2H_{\widetilde S}-3C)$, which is a contradiction. If $D$ is contracted to a point, then $-K_XD^2=(D|_{\widetilde S})^2 =  0 \neq -8$. Thus no $D$ is contracted and $\psi$ is small. \\

\item[Case 2:] Since $D$ is contracted by $\psi$, $0 = K_X^2D = (H-E)^2(aH-bE) = 6a+4$. So $D \sim 2H-3E$ and $D^2\vert_{\widetilde{S}} = (2H-3C)^3 = -10$. If $D$ is contracted to a curve, then $D$ is a conic bundle over a smooth curve $B$ \cite{JPR05}. Since $\widetilde S \in |-K_X|, D|_{\widetilde S} \sim a_1l_1 + \ldots + a_nl_n$ for integers $a_i$ and fibers $l_i$. But since $\rho(\widetilde S) = 2$, arguing as in Case 10, $n=1$. Then $(D|_{\widetilde{S}})^2 = -2a_1^2 \neq -10$. If $D$ is contracted to a point, then $-K_XD^2=0$ implies $(D|_{\widetilde{S}})^2 =0 \neq -10.$ Thus no $D$ is contracted and $\psi$ is small. \\

\item[Case 39:] Here $Y$ is the intersection of two quadrics in $\mathbb{P}^5$ with index $r = 2$. By Proposition 4.1 on \cite{ACM17}, there exists a smooth curve $C$ of degree 10 and genus 6 on $Y$ such that the blow-up of $Y$ along $C$ is a smooth weak Fano threefold of Picard number two. Using the same notation as above, $-K_X = 2H-E$, $E^3 = -30$, and $D\sim aH+bE$. Then $K_X^2D = 6a+10b$ so $D \sim 5H-3E$.  To show $|-K_X|$ gives a small contraction, consider $(D|_{\widetilde{S}})^2 = 25H^2 - 30HC + 9C^2 = -10$. Since $-10 \neq -2a^2$ for any integer $a$, $D$ cannot be contracted to a curve. In addition, as in Case 10, -10 is nonzero, thus $D$ is not  contracted to a point and $\psi$ is small.\\

\item[Case 67:] $Y$ is the intersection of two quadrics in $\mathbb{P}^5$ with index $r = 2$. By Proposition 4.1 in \cite{ACM17}, there exists a smooth curve $C$ of degree 8 and genus 3 on $Y$ such that the blow-up of $Y$ along $C$ is a smooth weak Fano threefold of Picard number two. Proceeding as in the previous cases, using $H^3 = 4$, $H^2E = 0$, $HE^2 = -8$, and $E^3 = -20$, we find $D \sim 3H - 2E$.   Using $(H_{\widetilde S})^2 = 8$, $H_{\widetilde S}C = 8$, and $C^2 = 4$, we find $(D|_{\widetilde S})^2 = -8$. As before, this implies that $D$ cannot be contracted to a point. If $D$ is contracted to a curve, then as in the previous cases, $D|_{\widetilde S}$ would be a multiple of a fibre of $D \to B$. However, in this case $\widetilde S$ does not contain any smooth rational curves, as such a curve would have self-intersection -2, and given $(H_{\widetilde S})^2 = 8$, $H_{\widetilde S}C = 8$, and $C^2 = 4$, we see that the square of any divisor class in $\widetilde S$ is divisible by $4$. Thus no $D$ is contracted and $\psi$ is small.\\

\item[Case 78:] By Proposition 6.1 in \cite{ACM17}, the blow up of the smooth curve $C$ with degree 4 and genus 0 in $S$ on $Y$ is a weak Fano threefold. Then $0 = K_X^2D =12a+6b$, so $D\sim H-2E$  and $(D|_{\widetilde{S}})^2 = -8$. If $D$ is contracted to a curve $B$, using similar arguments as in case 10, $l \sim \frac{1}{2}(H-2C)$ which is a contradiction. If $D$ is contracted to a point, then $0 = -K_XD^2 = (D|_{\widetilde{S}})^2 =-8$, a contradiction. Thus no $D$ is contracted and $\psi$ is small. \\

\item[Case 47:] By Proposition 3.1 in \cite{ACM17}, the blow up of the smooth curve $C$ with degree 12 and genus 11 in $S$ on $Y$ is a weak Fano threefold. To show $|-K_X|$ is small, consider $D \sim aH + bE$. Using $-K_X = 3H-E$ and $E^3 = -56$, compute $K_X^2D = 6a-16b$ so $D \sim 8H-3E$. Then $(D|_{\widetilde S})^2 = -12 \neq -2a_1^2$  so $D$ is not contracted to a curve. And -12 is certainly not 0 so $D$ is not contracted to a point. Thus no $D$ is contracted and $\psi$ is small.\\

\item[Case 72:] This is very similar to Case 39. Using $H^3 = 2$, $H^2E = 0$, $HE^2 = -10$, and $E^3 = -40$, we find $D \sim 5H - 2E$. Using $(H_{\widetilde S})^2 = 6$, $H_{\widetilde S}C = 10$, and $C^2 = 2g - 2 = 10$, we find $(D|_{\widetilde S})^2 = -10$. This leads to a contradiction as in Case 39.\\

\item[Case 102:] This is very similar to case 72. Using $-K_X = 3H-E, E^3 = -28$, and $D\sim aH+bE$, then $K^2_XD = 10a+20b$ so $D\sim 2H-E$. Then $(D|_{\widetilde S})^2 = -4.$ This leads to a contradiction as in case 39. 
\end{enumerate}

\end{proof}

\noindent \textbf{Remark 1:} The geometric realizability of the following E1-E1 cases still remain open: 28,59,61,80. See the next section for for why the methods used in Theorem 1 failed to produce links in these cases and possible alternative methods for future work. \\

In \cite{CM13}, one open case remained in the E2-E2 table. The following theorem shows this link is geometrically constructable.  
\begin{theorem}
The numerical invariants listed in the E2-E2 table, case number 3, is geometrically realizable. 
\end{theorem}

\begin{proof}
Pick a point $P \in Y = X_{10}$ not on any line, and let $\phi: X \to Y$ be the blowup of $P$. By \cite{Reid80}, Section 3, first Theorem, $|-K_X|$ is free and defines a small birational morphism $\psi: X \to X'$. By classification of smooth Fano threefolds, $X$ is not Fano since $-K_X^3 = 2$, so $X$ must be weak Fano. As this case is not on \cite{JPR07}, it is then of type E2-E2. 

\end{proof}

\noindent \textbf{Remark 2 - Two Corrections:} An error was corrected on the original E1-E2 table from \cite{CM13}. Case number 1 was originally listed as existing per \cite{Tak89}. This was a typo. This case has been updated to \textbf{not} exist by \cite{Kal09}, case number 24. 

In addition, a second error was corrected on the E2-E2 table, case number 2. Originally, this was marked with an ``x", citing \cite{Kal09}. This is now updated to exist per \cite{Kal09} (case number 22).\\

\noindent \textbf{Remark 3 - E5-E5 Update:} The only open numerical case on the E5-E5 table has been updated to show existence. Its construction can be found in Cheltsov and J. Park's, \textit{Sextic double solids} \cite{CP10}, Example 1.5. As a side note, in the same paper, Example 1.6 is a construction of the (already known to exist) E3/4 - E3/4 case number 2. \\

The tables listed in the appendix are now the most up-to-date as to the status of classifying weak Fano threefolds with small anticanonical contraction and Picard number two. \\

\newpage
\section{Tables}
The below tables provide the complete numerical classification of all smooth weak Fano threefolds with small anticanonical morphism and Picard number two, along with a reference to the geometric construction. There are four open cases remaining. Any numbers not appearing were on the original paper \cite{CM13} and have since been shown to not exist.

The notation used in the tables below are as follows: let $C$ be a smooth curve of genus $g$ and degree $d$ on a smooth Fano threefold $Y$ with $\rho(Y)=1$. Let $-K_Y$ be the canonical divisor on $Y$ and let $X$ be the blow up of $Y$ centered at $C$. The index of $Y$ and $Y^+$ are $r$ and $r^+$ respectively. The strict transform of a divisor $D \in \text{Pic} X$ across the flop $\chi$ is denoted $\widetilde{D}$. Write $\widetilde{E^+} = \alpha(-K_X) + \beta E$ and similarly let $\widetilde{E} = \alpha^+(-K_{X^+}) + \beta^+E^+$. The defect $e$ of the flop is defined as $E^3 - \widetilde{E}^3$. Define $d^+$, $g^+$, and $-K_{Y^+}$ similarly for the blow up of $Y^+$ centered at $C^+$ on the right side of the 2-ray diagram. 

\label{E1E1table}
\begin{table}[h]
\caption{E1-E1}
\begin{center}
\begin{tabular}{|c|c|c|c|c|c|c|c|c|c|c|c|c|c|c|}
\hline
\textit{No.} & $-K_X^3$ & $-K_Y^3$ & $-K_{Y^+}^3$ & $\alpha$ & $\beta$ & $r$ & $d$ & $g$ & $r^+$ & $d^+$ & $g^+$ & $e/r^3$ & Exist? & Ref \\ \hline \hline
$\textit{1.}$ & 2 & 6 & 6 & 3 & -1 & 1 & 1 & 0 & 1 & 1 & 0 & 47 & :) & \cite{Isk78} \\ \hline
$\textit{2.}$ & 2 & 8 & 8 & 4 & -1 & 1 & 2 & 0 & 1 & 2 & 0 & 88 & :) & Thm 3.1 \\ \hline
$\textit{3.}$ & 2 & 10 & 10 & 5 & -1 & 1 & 3 & 0 & 1 & 3 & 0 & 153 & :) & \cite{ACM17} \\ \hline
$\textit{4.}$ & 2 & 12 & 12 & 6 & -1 & 1 & 4 & 0 & 1 & 4 & 0 & 248 & :) & \cite{ACM17} \\ \hline
$\textit{5.}$ & 2 & 14 & 14 & 7 & -1 & 1 & 5 & 0 & 1 & 5 & 0 & 379 & :) & \cite{ACM17} \\ \hline
$\textit{6.}$ & 2 & 16 & 16 & 8 & -1 & 1 & 6 & 0 & 1 & 6 & 0 & 552 & :) & \cite{ACM17} \\ \hline
$\textit{7.}$ & 2 & 18 & 18 & 9 & -1 & 1 & 7 & 0 & 1 & 7 & 0 & 773 & :) & \cite{ACM17} \\ \hline
$\textit{8.}$ & 2 & 22 & 22 & 11 & -1 & 1 & 9 & 0 & 1 & 9 & 0 & 1383 & :) & \cite{ACM17} \\ \hline
$\textit{10.}$ & 2 & 10 & 10 & 4 & -1 & 1 & 4 & 1 & 1 & 4 & 1 & 56 & :) & Thm 3.1 \\ \hline
$\textit{11.}$ & 2 & 12 & 12 & 5 & -1 & 1 & 5 & 1 & 1 & 5 & 1 & 115 & :) & \cite{ACM17} \\ \hline
$\textit{12.}$ & 2 & 14 & 14 & 6 & -1 & 1 & 6 & 1 & 1 & 6 & 1 & 204 & :) & \cite{ACM17} \\ \hline
$\textit{13.}$ & 2 & 16 & 16 & 7 & -1 & 1 & 7 & 1 & 1 & 7 & 1 & 329 & :) & \cite{ACM17} \\ \hline
$\textit{14.}$ & 2 & 18 & 18 & 8 & -1 & 1 & 8 & 1 & 1 & 8 & 1 & 496 & :) &  \cite{ACM17} \\ \hline
$\textit{15.}$ & 2 & 22 & 22 & 10 & -1 & 1 & 10 & 1 & 1 & 10 & 1 & 980 & :) & \cite{ACM17} \\ \hline
$\textit{18.}$ & 2 & 16 & 16 & 6 & -1 & 1 & 8 & 2 & 1 & 8 & 2 & 160 & :) & \cite{ACM17} \\ \hline
$\textit{19.}$ & 2 & 18 & 18 & 7 & -1 & 1 & 9 & 2 & 1 & 9 & 2 & 279 & :) & \cite{ACM17} \\ \hline
$\textit{20.}$ & 2 & 22 & 22 & 9 & -1 & 1 & 11 & 2 & 1 & 11 & 2 & 649 & :) & \cite{ACM17}  \\ \hline
$\textit{23.}$ & 2 & 22 & 22 & 8 & -1 & 1 & 12 & 3 & 1 & 12 & 3 & 384 & :) &  \cite{ACM17} \\ \hline
$\textit{28.}$ & 2 & 16 & 16 & 8 & -1 & 2 & 3 & 0 & 2 & 3 & 0 & 69 & ? &  \\ \hline
$\textit{29.}$ & 2 & 24 & 24 & 12 & -1 & 2 & 5 & 0 & 2 & 5 & 0 & 223 & :) & \cite{CM13} \\ \hline
$\textit{30.}$ & 2 & 32 & 32 & 16 & -1 & 2 & 7 & 0 & 2 & 7 & 0 & 521 & :) & \cite{ACM17} \\ \hline
$\textit{31.}$ & 2 & 40 & 40 & 20 & -1 & 2 & 9 & 0 & 2 & 9 & 0 & 1011 & :) &  \cite{ACM17}  \\ \hline
$\textit{33.}$ & 2 & 24 & 24 & 10 & -1 & 2 & 6 & 2 & 2 & 6 & 2 & 114 & :) & \cite{CM13}  \\ \hline
$\textit{34.}$ & 2 & 32 & 32 & 14 & -1 & 2 & 8 & 2 & 2 & 8 & 2 & 328 & :) & \cite{ACM17} \\ \hline
$\textit{35.}$ & 2 & 40 & 40 & 18 & -1 & 2 & 10 & 2 & 2 & 10 & 2 & 710 & :) &  \cite{ACM17} \\ \hline
$\textit{37.}$ & 2 & 32 & 32 & 12 & -1 & 2 & 9 & 4 & 2 & 9 & 4 & 183 & :) & \cite{ACM17}  \\ \hline
\end{tabular}
\label{tb:E1E1Table1}
\end{center}
\end{table}

\begin{table}
\caption{E1-E1 (continued)}
\begin{center}
\begin{tabular}{|c|c|c|c|c|c|c|c|c|c|c|c|c|c|c|}
\hline
\textit{No.} & $-K_X^3$ & $-K_Y^3$ & $-K_{Y^+}^3$ & $\alpha$ & $\beta$ & $r$ & $d$ & $g$ & $r^+$ & $d^+$ & $g^+$ & $e/r^3$ & Exist? & Ref \\ \hline \hline

$\textit{38.}$ & 2 & 40 & 40 & 16 & -1 & 2 & 11 & 4 & 2 & 11 & 4 & 469 & :) &  \cite{ACM17}  \\ \hline
$\textit{39.}$ & 2 & 32 & 32 & 10 & -1 & 2 & 10 & 6 & 2 & 10 & 6 & 80 & :) & Thm 3.1 \\ \hline
$\textit{40.}$ & 2 & 40 & 40 & 14 & -1 & 2 & 12 & 6 & 2 & 12 & 6 & 282 & :) &  \cite{ACM17}  \\ 
$\textit{41.}$ & 2 & 32 & 32 & 8 & -1 & 2 & 11 & 8 & 2 & 11 & 8 & 13 & :) & \cite{ACM17}  \\ \hline
$\textit{42.}$ & 2 & 40 & 40 & 12 & -1 & 2 & 13 & 8 & 2 & 13 & 8 & 143 & :) &  \cite{ACM17}  \\ \hline
$\textit{44.}$ & 2 & 54 & 54 & 25 & -1 & 3 & 9 & 2 & 3 & 9 & 2 & 571 & :) & \cite{ACM17} \\ \hline
$\textit{45.}$ & 2 & 54 & 54 & 22 & -1 & 3 & 10 & 5 & 3 & 10 & 5 & 372 & :) & \cite{ACM17} \\ \hline
$\textit{46.}$ & 2 & 54 & 54 & 19 & -1 & 3 & 11 & 8 & 3 & 11 & 8 & 221 & :) & \cite{ACM17} \\ \hline
$\textit{47.}$ & 2 & 54 & 54 & 16 & -1 & 3 & 12 & 11 & 3 & 12 & 11 & 112 & :) & Thm 3.1\\ \hline
$\textit{48.}$ & 2 & 54 & 54 & 13 & -1 & 3 & 13 & 14 & 3 & 13 & 14 & 39 & :) & \cite{ACM17} \\ \hline
$\textit{49.}$ & 2 & 64 & 64 & 30 & -1 & 4 & 8 & 2 & 4 & 8 & 2 & 418 & :) & \cite{CM13}  \\ \hline
$\textit{50.}$ & 2 & 64 & 64 & 26 & -1 & 4 & 9 & 6 & 4 & 9 & 6 & 261 & :) & \cite{CM13} \\ \hline
$\textit{51.}$ & 2 & 64 & 64 & 22 & -1 & 4 & 10 & 10 & 4 & 10 & 10 & 146 & :) & \cite{CM13} \\ \hline
$\textit{52.}$ & 2 & 64 & 64 & 18 & -1 & 4 & 11 & 14 & 4 & 11 & 14 & 67 & :) & \cite{CM13}  \\ \hline
$\textit{54.}$ & 4 & 10 & 10 & 2 & -1 & 1 & 2 & 0 & 1 & 2 & 0 & 28 & :) & \cite{Tak89} \\ \hline
$\textit{55.}$ & 4 & 14 & 14 & 3 & -1 & 1 & 4 & 0 & 1 & 4 & 0 & 68 & :) & \cite{ACM17} \\ \hline
$\textit{56.}$ & 4 & 18 & 18 & 4 & -1 & 1 & 6 & 0 & 1 & 6 & 0 & 144 & :) & \cite{ACM17} \\ \hline
$\textit{57.}$ & 4 & 22 & 22 & 5 & -1 & 1 & 8 & 0 & 1 & 8 & 0 & 268 & :) & \cite{ACM17} \\ \hline
$\textit{59.}$ & 4 & 16 & 16 & 3 & -1 & 1 & 6 & 1 & 1 & 6 & 1 & 42 & ? &  \\ \hline
$\textit{61.}$ & 4 & 22 & 22 & 4 & -1 & 1 & 10 & 2 & 1 & 10 & 2 & 80 & ? &  \\ \hline
$\textit{63.}$ & 4 & 24 & 14 & 2.5 & -0.5 & 2 & 5 & 1 & 1 & 5 & 1 & 25 & :) & \cite{Isk78} \\ \hline
$\textit{64.}$ & 4 & 32 & 18 & 3.5 & -0.5 & 2 & 7 & 1 & 1 & 7 & 1 & 77 & :) & \cite{ACM17} \\ \hline
$\textit{65.}$ & 4 & 40 & 22 & 4.5 & -0.5 & 2 & 9 & 1 & 1 & 9 & 1 & 171 & :) &  \cite{ACM17} \\ \hline
$\textit{67.}$ & 4 & 32 & 32 & 6 & -1 & 2 & 8 & 3 & 2 & 8 & 3 & 40 & :) & Thm 3.1 \\ \hline
$\textit{68.}$ & 4 & 40 & 40 & 8 & -1 & 2 & 10 & 3 & 2 & 10 & 3 & 110 & :) &  \cite{ACM17} \\ \hline
$\textit{69.}$ & 4 & 40 & 40 & 6 & -1 & 2 & 12 & 7 & 2 & 12 & 7 & 18 & :) &  \cite{ACM17} \\ \hline
$\textit{70.}$ & 4 & 54 & 16 & 11/3 & -1/3 & 3 & 9 & 3 & 1 & 5 & 0 & 103 & :) & \cite{ACM17} \\ \hline
$\textit{71.}$ & 4 & 54 & 54 & 13 & -1 & 3 & 8 & 0 & 3 & 8 & 0 & 164 & :) & \cite{ACM17} \\ \hline
$\textit{72.}$ & 4 & 54 & 54 & 10 & -1 & 3 & 10 & 6 & 3 & 10 & 6 & 60 & :) & Thm 3.1 \\ \hline
$\textit{74.}$ & 4 & 64 & 12 & 2.75 & -0.25 & 4 & 9 & 7 & 1 & 3 & 0 & 45 & :) & \small{[Tak89,IP99]} \\ \hline
$\textit{75.}$ & 4 & 64 & 64 & 14 & -1 & 4 & 8 & 3 & 4 & 8 & 3 & 82 & :) & \cite{CM13} \\ \hline
$\textit{76.}$ & 4 & 64 & 64 & 10 & -1 & 4 & 10 & 11 & 4 & 10 & 11 & 20 & :) & \cite{JPR05}  \\ \hline
$\textit{77.}$ & 6 & 10 & 10 & 1 & -1 & 1 & 1 & 0 & 1 & 1 & 0 & 11 & :) & \cite{Isk78} \\ \hline
$\textit{78.}$ & 6 & 16 & 16 & 2 & -1 & 1 & 4 & 0 & 1 & 4 & 0 & 32 & :) & Thm 3.1 \\ \hline
$\textit{79.}$ & 6 & 22 & 22 & 3 & -1 & 1 & 7 & 0 & 1 & 7 & 0 & 89 & :) &  \cite{ACM17} \\ \hline
$\textit{80.}$ & 6 & 18 & 18 & 2 & -1 & 1 & 6 & 1 & 1 & 6 & 1 & 12 & ? &  \\ \hline
$\textit{81.}$ & 6 & 40 & 18 & 2.5 & -0.5 & 2 & 9 & 2 & 1 & 5 & 0 & 47 & :) &  \cite{ACM17} \\ \hline
$\textit{83.}$ & 6 & 40 & 40 & 6 & -1 & 2 & 8 & 0 & 2 & 8 & 0 & 82 & :) &  \cite{ACM17}  \\ \hline
$\textit{84.}$ & 6 & 32 & 32 & 4 & -1 & 2 & 7 & 2 & 2 & 7 & 2 & 17 & :) & \cite{ACM17} \\ \hline
$\textit{86.}$ & 6 & 54 & 22 & 8/3 & -1/3 & 3 & 8 & 1 & 1 & 8 & 1 & 48 & :) & \cite{ACM17} \\ \hline
$\textit{87.}$ & 6 & 54 & 12 & 5/3 & -1/3 & 3 & 10 & 7 & 1 & 2 & 0 & 14 & :) & [\cite{Tak89} \\ \hline
$\textit{88.}$ & 6 & 54 & 54 & 7 & -1 & 3 & 9 & 4 & 3 & 9 & 4 & 31 & :) & \cite{ACM17} \\ \hline
$\textit{89.}$ & 6 & 64 & 40 & 4.5 & -0.5 & 4 & 8 & 4 & 2 & 10 & 4 & 24 & :) & \cite{CM13}  \\ \hline
$\textit{90.}$ & 6 & 64 & 64 & 10 & -1 & 4 & 7 & 0 & 4 & 7 & 0 & 47 & :) & \cite{CM13}  \\ \hline
$\textit{91.}$ & 6 & 64 & 64 & 6 & -1 & 4 & 10 & 12 & 4 & 10 & 12 & 2 & x & \cite{CM13} \\ \hline
\end{tabular}
\label{tb:E1E1Table2}
\end{center}
\end{table}

\begin{table}
\caption{E1-E1 (continued)}
\begin{center}
\begin{tabular}{|c|c|c|c|c|c|c|c|c|c|c|c|c|c|c|}
\hline
\textit{No.} & $-K_X^3$ & $-K_Y^3$ & $-K_{Y^+}^3$ & $\alpha$ & $\beta$ & $r$ & $d$ & $g$ & $r^+$ & $d^+$ & $g^+$ & $e/r^3$ & Exist? & Ref \\ \hline \hline
$\textit{92.}$ & 8 & 14 & 14 & 1 & -1 & 1 & 2 & 0 & 1 & 2 & 0 & 10 & :) & \cite{Tak89} \\ \hline
$\textit{93.}$ & 8 & 22 & 22 & 2 & -1 & 1 & 6 & 0 & 1 & 6 & 0 & 36 & :) & \cite{ACM17} \\ \hline
$\textit{94.}$ & 8 & 40 & 16 & 1.5 & -0.5 & 2 & 9 & 3 & 1 & 3 & 0 & 12 & :) &  \cite{ACM17}  \\ \hline
$\textit{96.}$ & 8 & 40 & 40 & 4 & -1 & 2 & 8 & 1 & 2 & 8 & 1 & 28 & :) &  \cite{ACM17}  \\ \hline
$\textit{97.}$ & 8 & 54 & 18 & 5/3 & -1/3 & 3 & 8 & 2 & 1 & 4 & 0 & 20 & :) & \cite{ACM17}  \\ \hline
$\textit{98.}$ & 8 & 64 & 22 & 1.75 & -0.25 & 4 & 7 & 1 & 1 & 7 & 1 & 14 & :) & \cite{CM13}  \\ \hline
$\textit{99.}$ & 8 & 64 & 64 & 6 & -1 & 4 & 8 & 5 & 4 & 8 & 5 & 10 & :) & \cite{BL12}  \\ \hline
$\textit{100.}$ & 10 & 18 & 18 & 1 & -1 & 1 & 3 & 0 & 1 & 3 & 0 & 9 & :) &  \cite{ACM17}  \\ \hline
$\textit{101.}$ & 10 & 40 & 22 & 1.5 & -0.5 & 2 & 7 & 0 & 1 & 5 & 0 & 18 & :) &  \cite{ACM17}  \\ \hline
$\textit{102.}$ & 10 & 54 & 54 & 4 & -1 & 3 & 8 & 3 & 3 & 8 & 3 & 8 & :) & Thm 3.1  \\ \hline
$\textit{103.}$ & 10 & 64 & 32 & 2.5 & -0.5 & 4 & 7 & 2 & 2 & 5 & 0 & 9 & :) & \cite{CM13}  \\ \hline
$\textit{104.}$ & 12 & 22 & 22 & 1 & -1 & 1 & 4 & 0 & 1 & 4 & 0 & 8 & :) & \cite{ACM17} \\ \hline
$\textit{105.}$ & 12 & 54 & 40 & 7/3 & -2/3 & 3 & 7 & 1 & 2 & 7 & 1 & 7 & :) & \cite{ACM17}  \\ \hline
$\textit{106.}$ & 12 & 64 & 16 & 0.75 & -0.25 & 4 & 7 & 3 & 1 & 1 & 0 & 5 & :) & \cite{Isk78} \\ \hline
$\textit{107.}$ & 14 & 40 & 40 & 2 & -1 & 2 & 6 & 0 & 2 & 6 & 0 & 6 & :) & \cite{ChSh11}  \\ \hline
$\textit{108.}$ & 14 & 54 & 18 & 2/3 & -1/3 & 3 & 7 & 2 & 1 & 1 & 0 & 4 & :) & \cite{Isk78} \\ \hline
$\textit{109.}$ & 16 & 54 & 22 & 2/3 & -1/3 & 3 & 6 & 0 & 1 & 2 & 0 & 4 & :) & \cite{Tak89} \\ \hline
$\textit{110.}$ & 18 & 40 & 22 & 0.5 & -0.5 & 2 & 5 & 0 & 1 & 1 & 0 & 3 & :) & \cite{Isk78} \\ \hline
$\textit{111.}$ & 22 & 64 & 64 & 2 & -1 & 4 & 5 & 0 & 4 & 5 & 0 & 1 & :) & \cite{CM13} \\  \hline
\end{tabular}
\label{tb:E1E1Table3}
\end{center}
\end{table}
\clearpage
\newpage

\label{E1E2table}
\begin{table}[h]
\caption{E1-E2}
\begin{center}
\begin{tabular}{|c|c|c|c|c|c|c|c|c|c|c|c|}
\hline
\textit{No.} & $-K_X^3$ & $-K_Y^3$ & $-K_{Y^+}^3$ & $\alpha$ & $\beta$ & $r$ & $d$ & $g$ & $e/r^3$ & Exist? & Ref \\ \hline \hline
$\textit{1.}$ & 4 & 40 & 12 & 5/2 & -1/2 & 2 & 12 & 7 & 24 & x & \cite{Kal09} \\ \hline
$\textit{2.}$ & 6 & 24 & 14 & 3/2 & -1/2 & 2 & 4 & 0 & 16 & :) & \cite{Tak89} \\ \hline
$\textit{3.}$ & 14 & 64 & 22 & 3/4 & -1/4 & 4 & 6 & 0 & 6 & :) & \cite{Tak89} \\ \hline
\end{tabular}
\label{tb:E1E2Table}
\end{center}
\end{table}

\label{E1E3table}
\begin{table}[h]
\caption{E1-E3/E4}
\begin{center}
\begin{tabular}{|c|c|c|c|c|c|c|c|c|c|c|c|}
\hline
\textit{No.} & $-K_X^3$ & $-K_Y^3$ & $-K_{Y^+}^3$ & $\alpha$ & $\beta$ & $r$ & $d$ & $g$ & $e/r^3$ & Exist? & Ref \\ \hline \hline
$\textit{4.}$ & 10 & 64 & 12 & 0.75 & -0.25 & 4 & 8 & 6 & 5 & :) & \cite{CM13} \\ \hline
$\textit{5.}$ & 12 & 54 & 14 & 2/3 & -1/3 & 3 & 8 & 4 & 4 & :) & \cite{CM13} \\ \hline
$\textit{6.}$ & 14 & 32 & 16 & 0.5 & -0.5 & 2 & 4 & 0 & 4 & :) & \cite{CM13} \\ \hline
$\textit{7.}$ & 16 & 40 & 18 & 0.5 & -0.5 & 2 & 6 & 1 & 3 & :) & \cite{CM13} \\ \hline
\end{tabular}
\label{tb:E1E3Table}
\end{center}
\end{table}

\label{E1E5table}
\begin{table}[h]
\caption{E1-E5}
\begin{center}
\begin{tabular}{|c|c|c|c|c|c|c|c|c|c|c|c|}
\hline
\textit{No.} & $-K_X^3$ & $-K_Y^3$ & $-K_{Y^+}^3$ & $\alpha$ & $\beta$ & $r$ & $d$ & $g$ & $e/r^3$ & Exist? & Ref \\ \hline \hline
$\textit{3.}$ & 8 & 64 & 17/2 & 0.75 & -0.25 & 4 & 9 & 9 & 6 & :) & \cite{CM13} \\ \hline
$\textit{4.}$ & 10 & 24 & 21/2 & 0.5 & -0.5 & 2 & 3 & 0 & 6 & :) & \cite{CM13} \\ \hline
$\textit{5.}$ & 10 & 54 & 21/2 & 2/3 & -1/3 & 3 & 9 & 6 & 5 & :) & \cite{CM13} \\ \hline
$\textit{6.}$ & 12 & 32 & 25/2 & 0.5 & -0.5 & 2 & 5 & 1 & 5 & :) & \cite{CM13} \\ \hline
$\textit{7.}$ & 14 & 40 & 29/2 & 0.5 & -0.5 & 2 & 7 & 2 & 4 & :) & \cite{CM13} \\ \hline
\end{tabular}
\label{tb:E1E5Table}
\end{center}
\end{table}

\label{E2E2table}
\begin{table}[!h]
\caption{E2-E2}
\begin{center}
\begin{tabular}{|c|c|c|c|c|c|c|c|}
\hline
\textit{No.} & $-K_X^3$ & $-K_{Y}^3$ & $\alpha$ & $\beta$  & $e$ & Exist? & Ref \\ \hline
$\textit{1.}$ & 8 & 16 & 1 & -1 & 12 & :)  & \cite{Tak89} \\ \hline
$\textit{2.}$ & 4 & 12 & 2 & -1 & 30 & :) & \cite{Kal09} \\ \hline
$\textit{3.}$ & 2 & 10 & 4 & -1 & 90 & :) & Thm 3.2 \\ \hline
\end{tabular}
\label{tb:E2E2Table}
\end{center}
\end{table}

\label{E3E3table}
\begin{table}[h]
\caption{E3/4-E3/4}
\begin{center}
\begin{tabular}{|c|c|c|c|c|c|c|c|}
\hline
\textit{No.} & $-K_X^3$ & $-K_{Y}^3$ & $\alpha$ & $\beta$  & $e$ & Exist? & Ref \\ \hline
$\textit{1.}$ & 4 & 6 & 1 & -1 & 12 & :) & \cite{Kal09} \\ \hline
$\textit{2.}$ & 2 & 4 & 2 & -1 & 24 & :) & \cite{Puk88} \\ \hline
\end{tabular}
\label{tb:E3E3Table}
\end{center}
\end{table}

\label{E5E5table}
\begin{table}[h]
\caption{E5-E5}
\begin{center}
\begin{tabular}{|c|c|c|c|c|c|c|c|}
\hline
\textit{No.} & $-K_X^3$ & $-K_{Y}^3$ & $\alpha$ & $\beta$  & $e$ & Exist? & Ref \\ \hline
$\textit{1.}$ & 2 & 2.5 & 1 & -1 & 15 & :) & \cite{CP10} \\ \hline
\end{tabular}
\label{tb:E5E5Table}
\end{center}
\end{table}
\clearpage
\newpage

\section*{Acknowledgments}
We would like to thank our lifetime advisor and now colleague, V.V. Shokurov, for providing this classification question to the authors so long ago. It has inspired many years of fun and fruitful mathematics. In addition, the authors would like to thank V. Cheltsov for always motivating and encouraging the authors to continue in their work, and for his amazing ability to always make a global community feel local.



\newpage
\begin{thebibliography}{10}

\bibitem[ACM17]{ACM17} Arap, M., Cutrone, J., Marshburn, N.: \textit{On the existence of certain weak fano threefolds of picard number two}, Mathematica Scandinavica,[S.I.], v.120,n.1,p68-86, Feb. 2017 ISSN 1903-1807.

\bibitem[BL12]{BL12} Blanc, J., Lamy, S.: \textit{Weak Fano threefolds obtained by blowing-up a space curve and construction of Sarkisov links}. Proc. Lond. Math. Soc. (3) \textbf{105} (2012), no. 5, 1047--1075. 

\bibitem[BL15]{BL15} Blanc, J., Lamy, S.: \textit{On Birational Maps From Cubic Threefolds}. North-West. Eur. J. Math. 1 (2015), 55-84. 

\bibitem[CHNP12]{CHNP12} Corti, A., Haskins, M., Nordstr{\"o}m, J., Pacini, T.: \textit{G2-manifolds and associative submanifolds via semi-Fano 3-folds.} Duke Math. Journal, v.164, 1971-2092, 2012. 

\bibitem[CHNP13]{CHNP13} Corti, A., Haskins, M., Nordstr\"om, J., Pacini, T.: \textit{Asymptotically cylindrical Calabi-Yau 3-folds from weak Fano 3-folds.} Geom. Topol. \textbf{17} (2013), 1955--2059.

\bibitem[ChSh11]{ChSh11} I. Cheltsov and C. Shramov. \textit{Cremona groups and the icosahedron}. Monographs and Research
Notes in Mathematics. CRC Press, Boca Raton, FL, 2016.

\bibitem[CM13]{CM13} Cutrone, J.W., Marshburn, N.A.: \textit{Towards the classification of weak Fano threefolds with $\rho = 2$.} Cent. Eur. J. Math. \textbf{11} (2013), no. 9, 1552-1576.

\bibitem[CLM19]{CLM19}J. W. Cutrone, M. A. Limarzi, and N. A. Marshburn. \textit{A weak Fano threefold
arising as a blowup of a curve of genus 5 and degree 8 on} $\mathbb{P}^3$. Eur. J. Math.,
5(3):763–770, 2019.

\bibitem[CP10]{CP10} I. Cheltsov and J. Park, Sextic double solids. \textit{Cohomological and geometric approaches to rationality problems}, 75–132, Progr. Math., 282, Birkhauser Boston, Boston, MA, 2010.

\bibitem[Isk78]{Isk78} V.A. Iskovskikh. Fano 3-folds I, II. Math USSR, Izv. \textbf{11}, 485-527 (1977); 12, 469-506 (1978).

\bibitem[IP99]{IP99} Iskovskikh, V. A., Prokhorov, Yu. G.: \textit{Fano varieties. Algebraic
geometry V}, Encyclopaedia Math. Sci., \textbf{47}, Springer, Berlin, 1999.

\bibitem[JPR05]{JPR05} Jahnke, P., Peternell, T., Radloff, I.: \textit{Threefolds with big
and nef anti-canonical bundles I.} Math. Ann. \textbf{333} (2005), no. 3, 569-631.

\bibitem[JP06]{JP06} P. Jahnke, T. Peternell. \textit{Almost del Pezzo manifolds} Advances in Geom. , vol. 8, no. 3, 2008, pp. 387-411.

\bibitem[JPR07]{JPR07} Jahnke, P., Peternell, T., Radloff, I.: \textit{Threefolds with big
and nef anti-canonical bundles II.} Cent. Eur. J. Math. \textbf{9} (2011), no. 3, 449--488.

\bibitem[Kal09]{Kal09} A-S Kaloghiros. \textit{A classification of terminal quartic 3-folds and applications to rationality questions}. Math. Ann. \textbf{354}, 263–296 (2012).

\bibitem[Knu02]{Knu02} Knutsen, A.L.: \textit{Smooth Curves on Projective $K3$ surfaces.} Math. Scand. \textbf{90} (2002) 215--231.

\bibitem[Knu13]{Knu13} A.L. Knutsen: \textit{Smooth, isolated curves in families of Calabi-Yau 
threefolds in homogeneous spaces.} J. Korean Math. Soc. \textbf{50} (2013), No. 5, pp. 1033-1050.

\bibitem[Ko89]{Ko89} Koll\'ar, J.: \textit{Flops.} Nagoya Math. J. \textbf{113} (1989), 15-36.

\bibitem[Mi03]{Mi03} T. Minagawa: Global smoothing of singular weak Fano 3-folds. J. Math. Soc. Japan \textbf{55}, 695 - 711 (2003).

\bibitem[Mo82]{Mo82} Mori, S.: \textit{Threefolds whose canonical bundles are not numerically
effective.} Ann. of Math. (2) \textbf{116} (1982), no. 1, 133-176.

\bibitem[MM84]{MM84} Mori, S., Mukai, S.: \textit{Classification of Fano 3-folds with $B_2 \ge 2$. I.} Algebraic and topological theories (Kinosaki, 1984), 496-545, Kinokuniya, Tokyo, 1986.

\bibitem[Mu10]{Mu10} Mukai, S.: \textit{Curves and symmetric spaces, II.} Annals of Math., \textbf{172} (2010), 1539-1558.

\bibitem[Pa98]{Pa98} R. Paoletti: \textit{The Kahler cone in families of quasi-Fano threefolds}. Math. Z. \textbf{227}, 45-68 (1990).

\bibitem[Puk88]{Puk88} A.V. Pukhlikov. \textit{Birational Automorphisms of a Three-Dimensional Quartic with a Simple Singularity}, Mat. Sb. 135, 472-495 (1988) (Russian). [English transl.: Math. USSR-SB. 63 (1989) 457-482, Zbl. 668.14007]. \textit{Tokyo J. Math}., 12(2): 375-385, 1989.

\bibitem[Reid10]{Reid80} M. Reid, \textit{Lines on Fano 3-folds according to Shokurov}, Report 11 (1980) Mittag-Leffler Institute.

\bibitem[StD94]{StD94} B. Saint-Donat. \textit{Projective Models of K3 surfaces}, Ameri. J. Math. 96 (1974), 602-639.

\bibitem[Tak02]{Tak02} H. Takagi. On classification of $\mathbb Q$-Fano 3-folds of Gorenstein Index 2. I, II. \textit{Nagoya Math. J.}, 167:117-155,157-216, 2002.

\bibitem[Tak89]{Tak89} K. Takeuchi. \textit{Some birational maps of Fano 3-Folds} \textit{Compositio Math}.,71(3): 265-283, 1989. 	

\bibitem[Tak09]{Tak09} K. Takeuchi. \textit{Weak Fano 3-folds with del Pezzo fibration}, Eur. J. Math. 8 (2022), 1225–1290.

\bibitem[W92]{W92} P.M.H Wilson: \textit{The Kahler cone on Calabi-Yau threefolds}. Inv. Math. \textbf{107}, 561-583 (1992); \textit{Erratum}: \textbf{114}, 231-233 (1993).

\bibitem[W97]{W97} P.M.H. Wilson: \textit{Symplectic deformations of Calabi-Yau threefolds}. J. Diff. Geom. \textbf{45}, 611-637 (1997).

\bibitem[Zik23]{Zik23} Sokratis Zikas. \textit{Sarkisov links with centres space curves on smooth cubic surfaces}. \textit{Pulb. Mat., Barc.}, 67(2), 2023.
\end{thebibliography}
\end{document}